\documentclass[reqno]{amsart}

\usepackage[utf8x]{inputenc}

\usepackage{amsfonts}
\usepackage{accents}
\usepackage{amscd, amsmath, amsthm}
\usepackage{amssymb}
\usepackage[all]{xy}
\usepackage{amstext}
\usepackage{fancyhdr}
\usepackage[backref,letterpaper,pdftex=true,colorlinks,linkcolor=blue,citecolor=red,pdfstartview=FitH]{hyperref}
\usepackage{graphicx}%
\usepackage{enumerate}

\providecommand{\U}[1]{\protect\rule{.1in}{.1in}}
\providecommand{\U}[1]{\protect\rule{.1in}{.1in}}

\newtheorem{teorema}{Theorem}[section]

\newtheorem{lema}[teorema]{Lemma}
\newtheorem{corolario}[teorema]{Corollary}
\newtheorem{ejemplo}[teorema]{Example}

\numberwithin{equation}{section}

\newcommand{\Z}{\mathbb{Z}}

\newcommand{\C}{\mathbb{C}}

\def\cF{\mathcal F}

\begin{document}

\title{What is a true spectra of a finite Fourier transform}

\author{Javier Diaz-Vargas}
\address{Universidad Aut\'{o}noma de Yucat\'{a}n}
\email{javier.diaz@correo.uady.mx}
\author{Lev Glebsky}
\address{Universidad Aut\'onoma de San Luis Potos\'i}
\email{glebsky@cactus.iico.uaslp.mx}
\author{Carlos Jacob Rubio-Barrios}
\address{Universidad Aut\'{o}noma de Yucat\'{a}n}
\email{carlos.rubio@correo.uady.mx}

\thanks{This paper is in final form and no version of it will be submitted for
publication elsewhere.}

\subjclass[2010]{Primary 11T06; Secondary 13M10}
\keywords{}

\begin{abstract}
  In this  paper we deal with a finite abelian group $G$ and the abstract Fourier transform 
${\mathcal F}:\C^G\to \C^{\hat{G}}$. Then, we consider $\tilde{j}\circ \cF:\C^G\to \C^{\hat G}$ where 
$\tilde j:\C^{\hat{G}}\to \C^G$ is defined by the composition with a bijection $j:G\to \hat{G}$.
($\tilde j$ is a pullback of $j$.)   
 In particular, we  show that $(\tilde{j}\circ {\mathcal F})^2$ is a permutation if and only if $j$ is 
a group isomorphism. Then, we study how the spectra of $\tilde{j}\circ\cF$ depends on the isomorphism $j$.
\end{abstract}

\maketitle

\section{Introduction}

Let $G$ be a finite abelian group and $\hat G$ its dual. 
Abstractly, a Fourier transform is a linear operator $\cF:\C^G\to \C^{\hat G}$.
So, it is a linear map from one space to another and it is worthless to speak 
about its spectrum and period, etc. In order to do it, we have to identify
$\C^G$ with $\C^{\hat G}$. To this end, we may take a bijection 
$j:G\to \hat G$ and consider $\tilde j:\C^{\hat G}\to \C^G$ defined as
$(\tilde jf)(g)=f(j(g))$. A ``concrete'' Fourier transform is a composition
$\tilde j\circ\cF$.  
Our first result: $(\tilde j\circ\cF)^2$ is a permutation if and only if $j$ is an isomorphism.
In this case $(\tilde j\circ\cF)^2$ is (a pullback of) $\hat{j}^{-1}\circ j$ which is an automorphism of $G$.
Here, $\hat{j}:G\to \hat G$ is the dual of $j$.
This result distinguishes isomorphisms from other bijections by ``Fourier transform's point of view''.
 So, further we consider the case when $j$ is an isomorphism. 
But, which isomorphism $j$ should we take? From the pure group theoretic point of view, there are no natural choices
for $j$. Of course, we may try to find $j$ with $\hat{j}^{-1}\circ j$  equal to a fixed automorphism. Such a choice 
is not always possible and, if possible, is not unique. For example, if $G$ is an additive group of 
 $\Z/n\Z$, then $\hat{j}^{-1}\circ j$ is always the map $x\to -x$. Still, we show that, in this case, the spectra 
(the multiplicities of eigenvalues) of $\tilde{j}\circ\cF$ does depend on $j$.   
 If $G=(\Z/n\Z)^m$, then  there are more flexibility in the choice of  $\hat{j}^{-1}\circ j$.  
Particularly,  $\hat{j}^{-1}\circ j$ could be the identical transformation for even $m$.
We consider several other examples and show how the spectra may be calculated in those cases.

We do not give a receipt for choosing $j$. We just point out that there are several choices and the properties of
the Fourier transform do depend on these choices. 
In applied mathematics, for an additive group $\Z/n\Z$, which is a cyclic group $C_n$, up to an isomorphism, 
we choose $j(x)=w^x$, where
$w$ is the $n$-th root of unity closest to $1$. This corresponds to the choice of a generator of $C_n$.  
If $G=(\Z/n\Z)^m$, we may choose $j$ to be
the $m$-th power of the above defined $j$. Are those choices natural? 
Of course, the answer to this question depends on
the problem. If $G$ comes with presentation as a direct product (which is not unique), then there is a naturally 
defined presentation of $\hat G$ as a direct product. So, in this case, the choice of $j$ as a direct product might 
be natural. On the other hand, it is known that there are no natural isomorphism $G\to \hat G$ 
in the category of abelian 
groups. So, it looks like that for a pure group-theoretical problem
(without some additional structure) there is no natural choice for an isomorphism $j$. 
So, it is interesting to study the dependence of $\tilde j\circ\cF$ on the isomorphism $j$.

\section{Preliminaries}

Let $G$ be a finite abelian group. Set $\C^G=\{f:G\to \C\}$, the $\C$-valued functions on $G$. This is a $\C$-vector space of functions. Every $f\in \C^G$ can be expressed as a linear combination of the delta functions $\delta_g:G\to \{0,1\}$ defined by
$$\delta_g(x)=\left\{\begin{array}{ll}
1 & \mbox{if}\;\; x=g,\\
\phantom{-} & \phantom{-}\\
0 & \mbox{if}\;\; x\neq g,
\end{array}\right.$$
for every $g\in G$, as follows
$$f=\sum_{g\in G} f(g)\delta_g.$$
Indeed, evaluate both sides at each $x\in G$ and we get the same value. The functions $\delta_g$ span $\C^G$ and they are linearly independent: if $\sum_{g\in G} a_g\delta_g=0$, then evaluating the sum at $x\in G$ shows $a_x=0$. Thus, the functions $\delta_g$ are a basis of $\C^G$, so $\dim \C^G=|G|$.

A character of $G$ is a homomorphism $\chi:G\to S^1$. For a character $\chi$ on $G$, the conjugate character is the function $\overline{\chi}:G\to S^1$ given by $\overline{\chi}(g):=\overline{\chi(g)}$. Since for any complex number $z$ with $|z|=1$,  $\overline{z}=\frac{1}{z}$, we have that $\overline{\chi}(g)=\chi(g)^{-1}=\chi(g^{-1})$. The dual group, or character group, of $G$ is the set of homomorphisms $G\to S^1$ with the group law of pointwise multiplication of functions: $(\chi\psi)(g)=\chi(g)\psi(g)$. The dual group of $G$ is denoted by $\hat{G}$.

The following result is well known.

\begin{teorema}\label{teo1}
If $G$ is a finite abelian group, then $G$ and $\hat{G}$ are isomorphic.
\end{teorema}

The next theorem is the first step leading to an expression for each $\delta_g$ as a linear combination of characters of $G$, which will lead to a Fourier series expansion of $f$ (see \cite{Luong}).

\begin{teorema}[Orthogonality relations]
  Let $G$ be a finite abelian group. Then
  $$\sum_{g\in G} \chi(g)=\left\{\begin{array}{ll}
  |G| & \mbox{if}\;\; \chi=\chi_0,\\
  \phantom{-} & \phantom{-}\\
  0 & \mbox{if}\;\; \chi\neq \chi_0,
  \end{array}\right.\qquad\qquad
  \sum_{\chi\in \widehat{G}} \chi(g)=\left\{\begin{array}{ll}
  |G| & \mbox{if}\;\; g=1,\\
  \phantom{-} & \phantom{-}\\
  0 & \mbox{if}\;\; g\neq 1.
  \end{array}\right.$$
  \end{teorema}

The Fourier transform ${\mathcal F}:\C^G\to \C^{\hat{G}}$ is the linear map defined as
$${\mathcal F}(f)(\chi)=\frac{1}{\sqrt{|G|}}\sum_{g\in G} f(g)\chi(g)$$
where $\C^G$ and $\C^{\hat{G}}$ are $\C$-vector spaces of dimension $|G|=|\hat{G}|$.

The process of recovering $f$ from its Fourier transform ${\mathcal F}(f)$ is called Fourier inversion.
The following theorem is a direct corollary of the Orthogonality relations.
\begin{teorema}[Fourier inversion]
  Let $G$ be a finite abelian group. If $f\in \C^G$, then
  $$f(x)=\frac{1}{\sqrt{|G|}}\sum_{\chi\in\widehat{G}} {\mathcal F}(f)(\chi)\bar\chi(x)$$
for all $x\in G$.
\end{teorema}

Since ${\mathcal F}$ is an isomorphism of two different vector spaces, we are not allowed to talk about ${\mathcal F}^2$, the spectra of ${\mathcal F}$, etc.

On $\C^G$ there is a natural unitary scalar product $\langle \cdot,\cdot\rangle$ defined as follows
$$\langle f_1,f_2\rangle=\frac{1}{\sqrt{|G|}}\sum_{g\in G} \overline{f_1}(g)f_2(g).$$
With this scalar product, $\C^G$ and $\C^{\hat{G}}$ are unitary $\C$-vector spaces. 
The orthogonality relations imply that
$$\langle {\mathcal F}(f_1), {\mathcal F}(f_2)\rangle=\langle f_1,f_2\rangle$$
for all $f_1, f_2\in \C^G$. So, $\mathcal F$ is a unitary transform $\C^G\to\C^{\hat G}$.

\section{Main results of this paper}

Let us note, that $\C^G$ forms an algebra under pointwise addition and multiplication: 
$$
(f_1+f_2)(g)=f_1(g)+f_2(g),\;\;\;\;\;(f_1f_2)(g)=f_1(g)f_2(g).
$$
So, $\C^G$ contains multiplicative subgroup $(S^1)^G$. Naturally, $\hat G\subset (S^1)^G\subset \C^G$.
Let $j:G\to \hat{G}$ be a function and $\hat{j}:G\to \C^G$ be its dual function defined 
by $\hat{j}(h)(x)=\overline{j(x)(h)}$. Notice that $\hat{j}(G)\subseteq (S^1)^G$ and $G\to (S^1)^G$ is always a homomorphism (even if $j$ is not).

\begin{teorema}\label{teo2}
$\hat{j}(G)\subset \hat{G}$ if and only if $j$ is a homomorphism.
\end{teorema}

\begin{proof}
Suppose that $j$ is a homomorphism and let $g\in G$. Then, for all $x, y\in G$ we have
\begin{align*}
\hat{j}(g)(xy) &= \overline{j(xy)(g)}=\overline{(j(x)j(y))(g)}= \overline{j(x)(g)j(y)(g)}=\overline{j(x)(g)}\cdot \overline{j(y)(g)}\\
&= \hat{j}(g)(x)\hat{j}(g)(y)
\end{align*}
and hence $\hat{j}(G)\subset \hat{G}$.\\
Conversely, suppose that $\hat{j}(G)\subset \hat{G}$ and let $x,y,h\in G$. Then,
\begin{align*}
j(xy)(h) &= \overline{\overline{j(xy)(h)}}=\overline{\hat{j}(h)(xy)}\\
&= \overline{\hat{j}(h)(x)\hat{j}(h)(y)}\\
&= \overline{\overline{j(x)(h)j(y)(h)}}=(j(x)(h))(j(y)(h))=(j(x)j(y))(h),
\end{align*}
and therefore $j(xy)=j(x)j(y)$ for all $x, y\in G$, that is, $j$ is a homomorphism.
\end{proof}

\begin{corolario}\label{coro1}
$j$ is an isomorphism if and only if $\hat{j}(G)=\hat{G}$.
\end{corolario}

\begin{proof}
Suppose that $j$ is an isomorphism. By Theorem~\ref{teo2} we have 
$\hat{j}(G)\subset \hat{G}$ and $\hat{j}:G\to \hat{G}$. 
Moreover, $\hat{j}$ is a homomorphism.
We will prove that $\hat{j}$ is injective and thus it will be an isomorphism since $|G|=|\hat{G}|$ and $G$ is finite. Let $h\in\ker \hat{j}$. Then, $\hat{j}(h)(x)=1$ for all $x\in G$. It follows that $j(x)(h)=1$ for all $x\in G$ 
and, therefore, $\sum_{x\in G} j(x)(h)=|G|$. Since $j$ is a bijection, $j(x)$ runs over all $\hat{G}$ as $x$ runs over 
all $G$. By the orthogonality relations we obtain that $h$ is the identity of $G$. Therefore, $\hat{j}$ is injective.

Suppose now that $\hat{j}(G)=\hat{G}$. It follows that $\hat{j}:G\to\hat G$ is an isomorphism and we  deduce,
applying the above arguments to $\hat j$, 
that $\hat{\hat{j}}(G)=\hat G$ and $\hat{\hat{j}}$ is an isomorphism. We just have to check that $\hat{\hat{j}}=j$.
\end{proof}

Let $j:G\to \hat{G}$ be a bijection. Then, $j$ induces a unitary linear map $\tilde{j}:\C^{\hat{G}}\to \C^G$ defined as
$$\tilde{j}(\phi)(g)=\phi(j(g)).$$
Now, the composition $\tilde{j}\circ {\mathcal F}:\C^G\to \C^G$ is a unitary linear function. 
We will say that $P:\C^G\to \C^G$ is a permutation if $P(f)=f\circ p$ for all $f\in\C^G$, where $p:G\to G$ is 
a bijection.

\begin{teorema}\label{teo3}
Let $G$ be a finite abelian group, $j:G\to \hat{G}$ a bijection and $P=(\tilde{j}\circ {\mathcal F})^2$. 
Then, $P$ is a permutation if and only if $j$ is an isomorphism. In this case $p=\hat{j}^{-1}\circ j$ is an isomorphism.
\end{teorema}

\begin{proof}
Suppose that $P$ is a permutation. Then there is a bijection $p:G\to G$ such that $P(f)=f\circ p$ for all 
$f\in \C^G$. It is enough to prove that $j$ is a homomorphism since $j$ is a bijection.\\
Recall that $\delta_g:G\to \{0,1\}$ is the $g$-delta function ($g\in G$).
Then for every $x\in G$ we have
\begin{align*}
P(\delta_g)(x)&=\frac{1}{\sqrt{|G|}}\sum_{h\in G} {\mathcal F}(\delta_g)(j(h))\cdot (j(x))(h)= 
\frac{1}{|G|}\sum_{h\in G} (j(h))(g)\cdot (j(x))(h).
\end{align*}
Since $j$ is a bijection there exists an inverse map $j^{-1}:\hat G\to G$. Then, for every $\chi\in \hat{G}$ we have
\begin{align*}
P(f_g)(j^{-1}(\chi))=\frac{1}{|G|}\sum_{h\in G} j(h)(g)\chi(h)=\frac{1}{|G|}\sum_{h\in G} \overline{\hat{j}(g)(h)}\chi(h)=\frac{1}{\sqrt{|G|}}{\mathcal F}\left(\overline{\hat{j}(g)}\right)(\chi).
\end{align*}
On the other hand we have
$$
P(\delta_g)(j^{-1}(\chi))=\delta_g(p(j^{-1}(\chi)))=\left\{\begin{matrix}
1 & \mbox{if}\; g=p(j^{-1}(\chi)),\\
\phantom{-} & \phantom{-}\\
0 & \mbox{if}\; g\neq p(j^{-1}(\chi)).
\end{matrix}\right.
$$
It follows that for every $\chi\in\hat{G}$,
\begin{eqnarray}\label{ec1}
{\mathcal F}\left(\overline{\hat{j}(g)}\right)(\chi)=\left\{\begin{matrix}
\sqrt{|G|} & \mbox{if}\; g=p(j^{-1}(\chi)),\\
\phantom{-} & \phantom{-}\\
0 & \mbox{if}\; g\neq p(j^{-1}(\chi)).
\end{matrix}\right.
\end{eqnarray}
Now, by the Fourier inversion formula and relation~(\ref{ec1}) we have that for every $x\in G$
$$
\overline{\hat{j}(g)}(x)=\frac{1}{\sqrt{|G|}}\sum_{\chi\in \hat{G}} 
{\mathcal F}\left(\overline{\hat{j}(g)}\right)(\chi)\overline{\chi}(x)=\overline{\alpha}(x),
$$
where $\alpha=j(p^{-1}(g))$. Hence, $\hat{j}=j\circ p^{-1}$ and thus $\hat{j}(G)\subset \hat{G}$. 
Now, by Theorem~\ref{teo2} it follows that $j$ is a homomorphism as well as $p=\hat{j}^{-1}\circ j$.\\
Conversely, suppose that $j$ is an isomorphism and let $f\in \C^G$. Then, for every $g\in G$ we have
\begin{align*}
P(f)(g) &= (\tilde{j}\circ {\mathcal F})^2 (f)(g)= (\tilde{j}\circ {\mathcal F})(\tilde{j}({\mathcal F}(f))(g)=\tilde{j}({\mathcal F}(\tilde{j}({\mathcal F}(f)))(g)\\
&= {\mathcal F}(\tilde{j}({\mathcal F}(f)))(j(g))= \frac{1}{\sqrt{|G|}}\sum_{h\in G} \tilde{j}({\mathcal F}(f))(h)\cdot (j(g))(h)\\
&= \frac{1}{\sqrt{|G|}}\sum_{h\in G} {\mathcal F}(f)(j(h))\cdot (j(g))(h)\\
&= \frac{1}{\sqrt{|G|}}\sum_{h\in G} \frac{1}{\sqrt{|G|}}\sum_{l\in G} f(l)\cdot (j(h))(l)\cdot (j(g))(h)\\
&= \frac{1}{|G|}\sum_{l\in G}\sum_{h\in G} f(l)\cdot (j(h))(l)\cdot (j(g))(h)\\
&= \frac{1}{|G|}\sum_{l\in G}\sum_{h\in G} f(l)\overline{\hat{j}(l)(h)}j(g)(h)\\
&= \frac{1}{|G|}\sum_{l\in G} f(l)\sum_{h\in G} \overline{\hat{j}(l)}j(g)(h).
\end{align*}
From the orthogonality relations we have that
$$
\sum_{h\in G} \overline{\hat{j}(l)}j(g)(h)=\left\{\begin{matrix}
|G| & \mbox{if}\;\; \hat{j}(l)=j(g),\\
\phantom{-} & \phantom{-}\\
0 & \mbox{if}\;\; \hat{j}(l) \neq j(g).
\end{matrix}\right.
$$
Thus, for every $g\in G$, we have
$P(f)(g)=f(l)$
where $l=\hat{j}^{-1}(j(g))$.  This shows that $P$ is a permutation and $p=\hat{j}^{-1}\circ j$.
\end{proof}

\section{Examples}

\begin{ejemplo}\label{ex_1}

Consider the additive group $(\Z/n\Z)_{ad}$ of the ring $\Z/n\Z$. 
So, starting from this point we will denote the group operation
by $+$. The multiplication will  denote the ring multiplication. 
Consider the standard isomorphism $s:(\Z/n\Z)_{ad}\to \widehat{(\Z/n\Z)_{ad}}$ defined
by $s(x)(y)=e_n(xy)$ where $e_n(x)=e^{2\pi ix/n}$. Then any other
isomorphism $j:(\Z/n\Z)_{ad}\to \widehat{(\Z/n\Z)_{ad}}$ has the form $j=s\circ h$
for some isomorphism $h:(\Z/n\Z)_{ad}\to (\Z/n\Z)_{ad}$. The isomorphism $h$ is defined by its value $h(1)=l$ 
where $l$ and $n$ are relatively prime.
So, for all $x, y\in \Z/n\Z$
$$
j(x)(y)=s(h(x))(y)=s(xl)(y)=e_n(xyl).
$$
Now, since $j$ is an isomorphism, we have that $P=(\tilde{j}\circ {\mathcal F})^2$ is a permutation
with $p=\hat{j}^{-1}\circ j$ 
by Theorem~\ref{teo3}.
Calculations shows that $p(x)=-x$ for all such $j$. This shows that $P^2=(\tilde{j}\circ {\mathcal F})^4$ is the identity on $\C^{\Z/n\Z}$. Thus, the spectrum of $\tilde{j}\circ {\mathcal F}$ is a subset of the set $\{1, -1, i, -i\}$ 
of $4$th roots of unity. In the next example we show that the multiplicities  of eigenvalues do depend on $j$, but not too much. 
\end{ejemplo}

The following example is a particular case of Example~\ref{ex_1}.

\begin{ejemplo}
  Consider the additive group $(\Z/p\Z)_{ad}$ where 
  $p$ is an odd prime.  
To calculate the multiplicities of the eigenvalues in this case we use the multiplicative characters of 
$\Z/p\Z$. The point is that $\tilde j\circ {\mathcal F}$ decomposes on $2\times 2$ and $1\times 1$ 
matrix blocks in the
bases of multiplicative characters. It is well known but  we do the corresponding calculation here to point out 
the dependence on $j$, and that we may calculate the possible spectra of $\tilde j\circ {\mathcal F}$
without calculating Gauss sums.

A {\em multiplicative character} is a function $\psi:G\to \C$ such that $\psi(xy)=\psi(x)\psi(y)$ 
for all $x, y\in \Z/p\Z$ and $\psi(0)=0$.
Fix a generator $g$ of a multiplicative group of $\Z/p\Z$.
  Define $\psi_0, \psi_1,\ldots, \psi_{p-2}:G\to \C$ by setting $\psi_a(0)=0$, $\psi_a(g^b)=e^{2\pi iab/(p-1)}$. 
It is trivial to verify that $\psi_a$, $a=0,1,\ldots, p-2$ are multiplicative characters and account for all such.  Setting
  $$\delta_0(x)=\left\{\begin{array}{ll}
1 & \mbox{if}\;\; x=0,\\
\phantom{-} & \phantom{-}\\
  0 & \mbox{if}\;\; x\neq 0,
  \end{array}\right.$$
  we have that $\beta=(\delta_0, \alpha\psi_0, \alpha\psi_1,\ldots, \alpha\psi_{p-2})$ is an orthonormal ordered basis of $\C^G$, where $\alpha=1/\sqrt{p-1}$.

Let $j:G\to \widehat{G}$ be an isomorphism. Then, from Example~\ref{ex_1}, we have that $j:=j_l$ for some $l\in \{1,2,\ldots, p-1\}$, where $j_l(x)(y)=e_p(xyl)$. It follows that for every $f\in \C^G$ and for all $x\in G$,
\begin{align*}
\tilde{j}({\mathcal F}(f))(x)={\mathcal F}(f)(j_l(x))=\frac{1}{\sqrt{p}}\sum_{y\in G} f(y)j_l(x)(y)=\frac{1}{\sqrt{p}}\sum_{y\in G} f(y)e_p(xyl).
  \end{align*}

In particular, we have that
\begin{align*}
\tilde{j}({\mathcal F}(\delta_0))(x)=\frac{1}{\sqrt{p}}=\frac{1}{\sqrt{p}}\delta_0(x)+\frac{1}{\sqrt{p}}\psi_0(x)=\frac{1}{\sqrt{p}}\delta_0(x)+\frac{\sqrt{p-1}}{\sqrt{p}}\alpha\psi_0(x)
\end{align*}
for all $x\in G$.\\
Similarly, from the orthogonality relations we have that
\begin{align*}
\tilde{j}({\mathcal F}(\psi_0))(x) &=\left\{\begin{array}{ll}
-\frac{1}{\sqrt{p}} & \mbox{if}\;\; x\neq 0,\\
\phantom{-} & \phantom{-}\\
\frac{p-1}{\sqrt{p}} & \mbox{if}\;\; x=0,
\end{array}\right.\\
&=\frac{p-1}{\sqrt{p}}\delta_0(x) - \frac{1}{\sqrt{p}}\psi_0(x)
\end{align*}
and so
\begin{align*}
  \tilde{j}({\mathcal F}(\alpha\psi_0))(x)=\frac{\sqrt{p-1}}{\sqrt{p}}\delta_0(x)-\frac{1}{\sqrt{p}}\alpha\psi_0(x).
  \end{align*}

It follows, that the matrix of $\tilde{j}\circ {\mathcal F}$ relative to $\beta$ begins in the ``northwest'' with the $2\times 2$ matrix block
\begin{equation}\label{eq_1d}
\left(\begin{matrix}
  \frac{1}{\sqrt{p}} & \frac{\sqrt{p-1}}{\sqrt{p}}\\
  \phantom{-} & \phantom{-}\\
\frac{\sqrt{p-1}}{\sqrt{p}} & -\frac{1}{\sqrt{p}}
\end{matrix}\right).
\end{equation}

Now, let $\psi\neq \psi_0$ be a multiplicative character on $G$. Then, for all $x\in G$ we have that
\begin{align*}
\tilde{j}({\mathcal F}(\psi))(x)=\frac{1}{\sqrt{p}}\sum_{y\in G} \psi(y)e_p(xyl).
\end{align*}
If $x=0$, then $e_p(0yl)=1$ and $\tilde{j}({\mathcal F}(\psi))(0)=\frac{1}{\sqrt{p}}\sum_{y\in G} \psi(y)=0$ where the last equality follows by the orthogonality relations.

If $x\neq 0$, then
\begin{align*}
\tilde{j}({\mathcal F}(\psi))(x)=\frac{1}{\sqrt{p}}\sum_{y\in G} \psi(x^{-1}y)e_p(ly)=S_l(\psi)\cdot \overline{\psi}(x),
\end{align*}
where $S_l(\psi)=\frac{1}{\sqrt{p}}\sum_{y\in G} \psi(y)e_p(ly)$.

If $\psi\neq \overline{\psi}$, we have that $(\tilde{j}\circ {\mathcal F})(\psi)=S_l(\psi)\cdot \overline{\psi}$ and $(\tilde{j}\circ {\mathcal F})(\overline{\psi})=S_l(\overline{\psi})\cdot \psi$.

Relative to the pair $(\psi, \overline{\psi})$, we get a $2\times 2$ matrix block of the form
$$\left(\begin{matrix}
  0 & S_l(\overline{\psi})\\
  S_l(\psi) & 0
  \end{matrix}\right).$$
To determine the spectra of this matrix we must determine the product $S_l(\overline{\psi})\cdot S_l(\psi)$. We have that
\begin{align*}
  S_l(\psi)\cdot S_l(\overline{\psi}) &=\frac{1}{p}\sum_{x,y\in G} \psi(x)e_p(lx)\overline{\psi}(y)e_p(ly)=\frac{1}{p}\sum_{x,y\in G} e_p(l(x+y))\psi(x)\overline{\psi}(y)\\
  &=\frac{1}{p}\sum_{a\in G} e_p(la)\cdot \sum_{y\in G} \psi(a-y)\overline{\psi}(y)\\
  &=\frac{1}{p}\sum_{y\in G} \psi(-y)\overline{\psi}(y) + \frac{1}{p}\sum_{a\in G^\times} e_p(la)\sum_{y\in G} \psi(a-y)\overline{\psi}(y).
\end{align*}
Since $\psi(0)=0$, for the first sum we have that
$$\sum_{y\in G} \psi(-y)\overline{\psi}(y)=\sum_{y\in G^\times} \psi(-y)\psi(y)^{-1}=\sum_{y\in G^\times} \psi(-1)=(p-1)\psi(-1).$$
From the orthogonality relations, it follows that
$$\sum_{y\in G} \psi(a-y)\overline{\psi}(y)=\sum_{y\in G^\times} \psi\left(\frac{a}{y}-1\right)=-\psi(-1)$$
since for $y\in G^\times$, $x=\frac{a}{y}-1$ if and only if $y=\frac{a}{x+1}$ for $x\in G-\{-1\}$.\\
Thus,
$$S_l(\psi)\cdot S_l(\overline{\psi})=\frac{p-1}{p}\psi(-1)-\frac{\psi(-1)}{p}\sum_{a\in G^\times} e_p(la)=\psi(-1)\left(\frac{p-1}{p}+\frac{1}{p}\right)=\psi(-1),$$
since $\sum_{a\in G^\times} e_p(la)=-1$ by the orthogonality relations. It follows that $S_l(\psi)\cdot S_l(\overline{\psi})=\pm 1$ because of $1=\psi(1)=(\psi(-1))^2$. Since $\psi_k(g^r)=e^{2\pi i rk/(p-1)}$ and $g^{(p-1)/2}=-1$, we have that $\psi(-1)=-1$ for $\frac{p-1}{2}$ nontrivial multiplicative characters and $\psi(-1)=1$ for $\frac{p-1}{2}-1=\frac{p-3}{2}$ nontrivial multiplicative characters.

If $\psi=\overline{\psi}$, then $\psi(x)=\psi_{\frac{p-1}{2}}=(\frac{x}{p})$ is the Legendre symbol of $x$ mod $p$. 
In this case, $S_l(\psi)=\frac{1}{\sqrt{p}}\sum_{y\in G} e_p(ly)(\frac{y}{p})=\frac{1}{\sqrt{p}}\sum_{y\in G} e_p(ly^2)$ 
where the last equality follows by Theorem 4.17 of~\cite{Nathanson}.
From the consideration above, we know that 
$$
S_l(\psi_\frac{p-1}{2})^2=\left(\frac{-1}{p}\right) = \left\{\begin{array}{rl}
1  & \mbox{if}\;\; p\equiv 1\pmod{4},\\
\phantom{-} & \phantom{-}\\
-1 & \mbox{if}\;\; p\equiv 3\pmod{4}.\\
\end{array}\right.
$$
Thus,
\begin{align}\label{rel1}
S_l(\psi_\frac{p-1}{2})=\left\{\begin{array}{rl}
\pm 1 & \mbox{if}\;\; p\equiv 1\pmod{4} ,\\
\phantom{-} & \phantom{-}\\
\pm i & \mbox{if}\;\; p\equiv 3\pmod{4}.\\
\end{array}\right.
\end{align}
So, we just need to choose the proper sign. Notice that $S_l(\psi)=\overline{\psi(l)}S_1(\psi)$.
It follows that any sign in Eq.\ref{rel1} is possible depending on whether $l$ is a quadratic residue mod $p$ or not.
We know that all two possibilities happen, but we are not able to decide which one corresponds to which $l$
without evaluating the Gauss sum. The evaluation of the Gauss sum is less elementary than the calculations used here.      

Finally, since $(\tilde{j}\circ {\mathcal F})(\psi)=S_l(\psi)\cdot \psi$, relative to the pair $(\psi,\psi)$ 
we get a $1\times 1$ matrix of the form $(S_l(\psi))$.

It follows that the characteristic polynomial of $\tilde{j}\circ {\mathcal F}$ is
$$
(\lambda^2-1)\prod_{j=1}^{\frac{p-3}{2}}(\lambda^2-\psi_j(-1))(\lambda\pm 1),\;\;\mbox{if}\;\;p\equiv 1\mod 4
$$
and 
$$
(\lambda^2-1)\prod_{j=1}^{\frac{p-3}{2}}(\lambda^2-\psi_j(-1))(\lambda\pm i),\;\;\mbox{if}\;\;p\equiv 3\mod 4
$$
where the choice of signs depend on $l$ and, consequently, $j$.

\end{ejemplo}
Let $R_4=\{1,-1,i,-i\}$ be a group of $4$th roots of unity and $\Z[R_4]$ its group algebra. 
We write an element $a\in\Z[R_4]$ as $a=a_1+a_{-1}[-1]+a_i[i]+a_{-i}[-i]$. Where we suppose that $[1]=1$.
The sum in $\Z[R_4]$ is defined by 
\begin{align*}
(a_1+a_{-1}[-1]+a_i[i]+a_{-i}[-i])+(b_1+b_{-1}[-1]+b_i[i]+b_{-i}[-i])=\\
(a_1+b_1)+(a_{-1}+b_{-1})[-1]+(a_i+b_i)[i]+(a_{-i}+b_{-i})[-i] 
\end{align*}
and the multiplication is the prolongation of the multiplication in $R_4$ ($[\alpha][\beta]=[\alpha\beta]$) 
by linearity.
We say that $a\in\Z[R_4]$ represents the spectrum of an operator $A$ ($A^4=id$) if $a_\alpha$ is the multiplicity
of the eigenvalue $\alpha$. Let $s=1+[-1]+[i]+[-i]$

\begin{teorema}\label{th_Zp}
  Let $p>2$ be a prime.
 
If $p\equiv 1 \mod 4$ then either $\frac{p-1}{4}s+1$ or $\frac{p-1}{4}s+[-1]$
represents the spectrum of $\tilde{j}\circ {\mathcal F}:\C^{\Z_p}\to \C^{\Z_p}$. Both cases are possible
depending on $j$. 

If $p\equiv 3 \mod 4$ then either $\frac{p+1}{4}s-[i]$ or $\frac{p+1}{4}s-[-i]$
represents the spectrum of $\tilde{j}\circ {\mathcal F}:\C^{\Z_p}\to \C^{\Z_p}$. Both cases are possible
depending on $j$.  
\end{teorema} 

\begin{ejemplo}
  Consider the abelian group $G=(\Z/n\Z)^m$. Then, the automorphism group of $G$ is $GL_m(\Z/n\Z)$, the group
of invertible $m\times m$ matrices in $\Z/n\Z$. Let $\langle\cdot,\cdot\rangle:G\times G\to \Z/n\Z$ be the
natural scalar product on $G$:
$$\langle h,g\rangle=h_1g_1+h_2g_2+\dots +h_mg_m.$$  
For an isomorphism $j:G\to\hat{G}$, there is $M_j\in GL_m(\Z/n\Z)$ such that 
  \begin{align*}
    j(g)(h) &
    = e_n(\langle h, M_j g\rangle).
  \end{align*}
We can check that $M_{\hat{j}}=-M^t_j$ where $M^t_j$ denotes the transpose matrix of $M_j$. 
It follows that, in this case, $p(g)=-(M^t_j)^{-1}M_jg$.
\begin{lema}
Let $M_{j'}=T^tM_jT$ for some $T\in GL_m(\Z/n\Z)$. Then $\tilde{j'}\circ {\mathcal F}$ is unitary 
equivalent to $\tilde{j}\circ {\mathcal F}$. The unitary equivalence comes
from base change in $(\Z/n\Z)^m$ by $T$.
\end{lema} 
\end{ejemplo}

  If $M^t=M$, then $p(g)=-g$ for all $g\in G$ and, hence, $P^2=id$.
  As in Example~\ref{ex_1} it follows that the set of eigenvalues of $\tilde{j}\circ {\mathcal F}$ 
is a subset of $R_4$. Let $n=p>2$ be a prime. Then it is not hard to calculate the spectrum of 
$\tilde{j}\circ {\mathcal F}$. Indeed, if $M\in GL_m(\Z/p\Z)$ is symmetric ($M^t=M$), then there
exists $T\in GL_m(\Z/p\Z)$ such that
$$
M=T^t\left( \begin{array}{cc}\begin{array}{cc}1 & \\
                                                & 1 \end{array} & 0\\
                     0 & \begin{array}{cc}\ddots & \\
                                              & l \end{array}
           \end{array}\right)T,        
$$  
(see \cite{nref}, Theorem 9.4). So, changing the base in $(\Z/p\Z)^m$, we get that $\tilde j\circ {\mathcal F}$ is isomorphic to
a tensor product  of $m$  Fourier transforms of $\Z/p\Z$, all but one with $l=1$.
 We need the following lemma.
\begin{lema}
Let $s=1+[-1]+[i]+[-i]$, $k\in \Z$ and $\alpha\in R_4$.
\begin{itemize}
\item Let $a,b\in \Z[R_4]$ represent the spectrum of $A$ and $B$, respectively. Then $ab$ represents the spectrum of
$A\otimes B$.
\item $s\cdot [\alpha]=s$.
\item $s^m=4^{m-1}s$.
\item $(ks+[\alpha])^m=\frac{(4k+1)^m-1)}{4}s+[\alpha^m]$.
\item $(ks-[\alpha])^m=\frac{(4k-1)^m-(-1)^m}{4}s+(-1)^m[\alpha^m]$.
\end{itemize}
\end{lema}    
This lemma with Theorem~\ref{th_Zp} imply the following corollary.
\begin{corolario}
Let $p>2$ be a prime and the isomorphism $j:(\Z/p\Z)^m\to (\Z/p\Z)^m$ be symmetric (with $M_j^t=M_j$).

If $p\equiv 1 \mod 4$ then $\frac{p^m-1}{4}s+[\pm 1]$ represents the spectrum of 
$\tilde{j}\circ{\mathcal F}$. Both cases are possible, depending on the choice of $j$.

If $p\equiv 3 \mod 4$ then $\frac{p^m-(-1)^m}{4}s+(-1)^m[\pm (i)^m]$ represents the spectrum of 
$\tilde{j}\circ{\mathcal F}$. Both cases are possible, depending on the choice of $j$. 
\end{corolario}

  In a similar way, if $M^t=-M$ then $p=id$ for all $g\in G$. In this case, the eigenvalues of 
$\tilde{j}\circ {\mathcal F}$ are $\pm 1$. 

In principle, using multiplicative characters we can calculate spectra of different Fourier transforms
on $(\Z/p\Z)^m$. As multiplicative characters with $\delta_0$ form a basis on $\C^{\Z/p\Z}$, the tensor products of
these characters form a basis on $\C^{(\Z/p\Z)^m}$. We show how to use it by an example.
First of all, we remind that if $f_1,f_2\in\C^{\Z/p\Z}$, then 
$f_1\otimes f_2\in\C^{(\Z/p\Z)^2}$ is defined as $(f_1\otimes f_2)(x_1,x_2)=f_1(x_1)f_2(x_2)$. 

\begin{ejemplo}
Let $G=(\Z/p\Z)^2$ and $j:G\to G$ be defined by 
$$
M_j=\left(\begin{array}{cc} 1 & 0\\
                            1 & 1
             \end{array}\right).
$$
Notice that $M_j$ is not symmetric. In this case we may find a transformation $T\in GL_2(\Z/p\Z)$ such that
$$
\left(\begin{array}{cc} 1 & 0\\
                            1 & 1
             \end{array}\right) =
T\left(\begin{array}{cc} 0 & 1\\
                            k & 0
             \end{array}\right)T^t,
$$
where $k=\frac{1-\sqrt{-3}}{2}$. Generally speaking, $k$ is not necessarily in $\Z/p\Z$ but may belong to its 
quadratic extension. Let $p\equiv 1 \mod 3$. Then $k\in\Z/p\Z$. We only consider this case. The other case
require additional considerations.
So, let $p\equiv 1 \mod 3$. Then by changing the base in $(\Z/p\Z)^2$, we obtain
$$
\tilde M_j=\left(\begin{array}{cc} 0 & 1\\
                            k & 0
             \end{array}\right).
$$
So, we have to study $\tilde{j}\circ {\mathcal F}$ with $j$ defined by $\tilde M_j$.
Let $\psi_1, \psi_2$ be nontrivial multiplicative characters on $\Z/p\Z$. Then 
$(\tilde{j}\circ{\mathcal F})(\psi_1\otimes\psi_2)=
S(\psi_1)S(\psi_2)\overline{\psi_2(k)}(\overline{\psi_2}\otimes\overline{\psi_1})$. 
So, for non-trivial multiplicative characters, we have the following two cases:
\begin{description}
\item[I] $\psi_1\neq \overline{\psi_2}$. Then the restriction of $\tilde{j}\circ\cF$ on subspace
spanned by \\ $\psi_1\otimes\psi_2,\;\overline{\psi_2}\otimes\overline{\psi_1}$ is 
$$
\left( \begin{array}{cc} 0 & S(\overline{\psi_1})S(\overline{\psi_2})\psi_1(k)\\
        S(\psi_1)S(\psi_2)\overline{\psi_2}(k) & 0
           \end{array}\right).
$$ 
The corresponding eigenvalues are $\lambda=\pm\sqrt{\psi_1(-k)\overline{\psi_2}(-k)}$.
\item[II] For nontrivial multiplicative characters $\psi$, the vectors $\psi\otimes\overline{\psi}$ are eigenvectors
with eigenvalues $\lambda=\psi(-k)$.
\end{description}  
We are left with the study of tensor products involving $\delta_0$ and trivial multiplicative character $\psi_0$.
Let $v_1=\delta_0+\frac{\sqrt{p}-1}{\sqrt{p-1}}\psi_0$ and $v_{-1}=\delta_0-\frac{\sqrt{p}+1}{\sqrt{p-1}}\psi_0$.
Notice that $v_{\alpha}$ is an $\alpha$-eigenvector of one-dimensional Fourier transform, see Eq.\ref{eq_1d}.
\begin{description}  
\item[III] Let $\psi$ be a nontrivial multiplicative character. On subspaces spanned by\\
$v_\alpha\otimes\psi, \overline{\psi}\otimes v_\alpha$, the $\tilde{j}\circ\cF$ acts as
$$
\alpha\left(\begin{array}{cc} 0 & S(\overline{\psi})\\
                 S(\psi)\overline{\psi}(k) & 0
             \end{array} \right).   
 $$
The corresponding eigenvalues are $\lambda=\pm\sqrt{\overline{\psi}(-k)}$.
\item[IV] The restriction of $\tilde{j}\circ\cF$ on $1$-dimensional subspace spanned by $v_\alpha\otimes v_\alpha$, 
is the identity operator, that is, $\lambda=1$.
\item[V]  The restriction of $\tilde{j}\circ\cF$ on subspace spanned by 
$v_1\otimes v_{-1},\; v_{-1}\otimes v_1$ is
$$
\left(\begin{array}{cc} 0 & -1\\
                      -1  & 0
     \end{array} \right).
$$ 
The corresponding eigenvalues are $\lambda=\pm 1$.
\end{description} 
Notice that $(-k)^3=1$. It follows that there are
\begin{itemize}
\item $\frac{p-4}{3}$ nontrivial multiplicative characters $\psi$ on $\Z/p\Z$ with $\psi(-k)=1$;
\item $\frac{p-1}{3}$ nontrivial characters with $\psi(-k)=e^{\frac{2\pi i}{3}}$;
\item $\frac{p-1}{3}$ nontrivial characters with $\psi(-k)=e^{-\frac{2\pi i}{3}}$.
\end{itemize}
\end{ejemplo}
Based on the above considerations, we get the following theorem.
\begin{teorema}
Let $p>2$ be a prime and $j:(\Z/p\Z)^2\to (\Z/p\Z)^2$ be an isomorphism with
$$
M_j=\left(\begin{array}{cc} 1 & 0\\
                            1 & 1
             \end{array}\right).
$$
Then, the eigenvalues of $\tilde{j}\circ\cF$ are the $6$th roots of unity. If $p\equiv 1 \mod 3$ and $p\geq 7$, then 
the multiplicities of the eigenvalues are the following:

\

\centering 
\begin{tabular}{rl}
\hline
\hline                        
$\lambda$ & \hspace{10.mm} multiplicity \\ [0.5ex]
\hline
$+1$ & $\frac{1}{2}\frac{(p-4)(p-7)}{9}+\frac{(p-1)^2}{9}+p-1$ \\[0.5ex]
$-1$ &  $\frac{1}{2}\frac{(p-4)(p-7)}{9}+\frac{(p-1)^2}{9}+2\frac{p-4}{3}+1$ \\ [0.5ex]
$e^{\frac{2\pi i}{6}}$ &  $\frac{3}{2}\frac{(p-4)(p-1)}{9}+2\frac{p-1}{3}$ \\ [0.5ex]
$e^{-\frac{2\pi i}{6}}$ &  $\frac{3}{2}\frac{(p-4)(p-1)}{9}+2\frac{p-1}{3}$ \\ [0.5ex]
$-e^{\frac{2\pi i}{6}}$ &  $\frac{3}{2}\frac{(p-4)(p-1)}{9}+p-1$ \\ [0.5ex]
$-e^{-\frac{2\pi i}{6}}$ &  $\frac{3}{2}\frac{(p-4)(p-1)}{9}+p-1$ \\ [1ex]
\hline
\end{tabular}
\end{teorema}

\end{document}